\documentclass[a4paper, 11pt,reqno]{amsart}
\usepackage[T1]{fontenc}
\usepackage[utf8]{inputenc}

\usepackage{amssymb,latexsym,amsmath,epsfig,amsthm,mathtools,tikz,hyperref}

\newtheorem{theorem}{Theorem}
\newtheorem{proposition}{Proposition}

\theoremstyle{definition}
\newtheorem{remark}{Remark}

\newcommand{\N}{\mbox{${\mathbb N}$}}
\newcommand{\Z}{\mbox{${\mathbb Z}$}}

\title{Rainbow Trapezoids with Given Area}

\author{Sukumar Das Adhikari}
\address{Department of Mathematics,
  Ramakrishna Mission Vivekananda Educational and Research Institute, India}
\email{adhikarisukumar@gmail.com}

\author{T\'assio Naia}
\address{Centre de Recerca Matem\`atica, Barcelona, Spain}
\email{tnaia@crm.cat}

\author{Oriol Serra}
\address{Department of Mathematics, Universitat Polit\`ecnica de  Catalunya, Barcelona, Spain}
\email{oriol.serra@upc.edu}

\thanks{A large part of this work was done while the first author was
  visiting the Department of Mathematics, UPC, Barcelona and he would
  like to thank the Institute for its hospitality. The second author
  was supported by the European Union's Horizon Europe Marie
  Sk\l{}odowska-Curie grant PARTIORI, project number 101106032
  \raisebox{-.2ex}{\includegraphics[height=2.2ex]{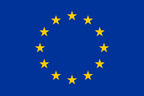}}. The
  third author acknowledges support from the Spanish Research Agency
  under project PID2023-147202NB-I00 funded by
  MICIU/AEI/10.13039/501100011033.}

\date{February 2026}

\begin{document}

\begin{abstract}
  A well-known result by Graham in Euclidean Ramsey Theory states
  that, for every positive real number $A$, every coloring of the
  plane with finite number of colors contains a monochromatic triangle
  of area $A$. We consider canonical versions of this result. We show
  that every $3$-coloring of the plane integer lattice contains either
  a rainbow triangle of area $1/2$ or a monochromatic rectangle of any
  given area whose sides are parallell to the axes. We also show that,
  under natural conditions, there are numbers $A$ and $B$ such that
  every coloring of the plane integer lattice contains either a
  monochromatic rectangle of area $A$ or a rainbow trapezoid of area
  $B$. As usual, only vertex colors are considered: e.g., a
  monochromatic rectangle is a set of four points in the lattice which
  a) are the vertices of a rectangle and b) are assigned the same
  color.


\end{abstract}

\maketitle

\section{Introduction}

\subsection{Euclidean Ramsey Theory}
The subject of Euclidean Ramsey Theory was developed in the pioneering papers
(\cite{EGMRSS}, \cite{EGMRSS1}, \cite{EGMRSS2}) of
Erd\H{o}s, Graham, Montgomery, Rothschild, Spencer and Straus.
Here, given a finite set $C\subset E^n$, where $E^n$ is the $n$ dimensional Euclidean space,
we shall say that ${\mathcal R}(C,n,r)$
holds if, for any $r$-coloring of $E^n$, there exists a monochromatic set
$C'\subset E^n$ such that $C'$ is congruent to $C$.
By saying that $C'$ is \emph{congruent} to $C$, one means that $C'$ is the
image of $C$ under some element in the group of Euclidean motions of $E^n$.

Given a finite subset $C$ in some Euclidean space, one says
that $C$ is \emph{ Ramsey} if for every positive integer
$r$, there is a positive integer $n=n(C,r)$ such that ${\mathcal R}(C,n,r)$
holds. Depending on the given finite Ramsey set $C$, one has to go for a suitably large $n$.
For example, let $P$ denote a set of two points at a distance $1$ apart.
Then  it is easy to see that ${\mathcal R}(P,1,2)$ does not hold.
Whereas, in dimension $2$, considering $2$-colorings of the vertices of
an equilateral triangle with sides of length $1$, we can
observe that ${\mathcal R}(P,2,2)$ holds.
For the early results in this area and further information, we would like to refer to~\cite{gra1} and
\cite{Gr2}.

\subsection{Rainbow structures}
Let $\chi$ be a coloring of a set $X$.
We say a subset $Y\subset X$ is \emph{rainbow} under~$\chi$ if $\bigl|\chi(Y)\bigr|=\min \bigl\{|\chi(X)|, |Y|\bigr\}$.
In other words, $Y$ is rainbow if either it has no two elements with the same color, or if all colors in $\chi(X)$ are
present in~$Y$.

The study of rainbow arithmetic structures was introduced by Jungi\'c, Licht, Mahdian, Ne\v set\v ril and~Radoi\v ci\'c~\cite{JLMNR2001}
(see also one of the early articles in the area~\cite{jnr}).
The existence of such rainbow structures requires some properties of the coloring.
By settling  a conjecture in~\cite{JLMNR2001},
Axenovich and Fon der Flaas~\cite{af2004} proved that every $3$-coloring of $[N]$  such that every color class has at least
$(N+4)/6$ colors contains a rainbow $3$-term arithmetic progression.  More precisely,

\begin{theorem}[{Axenovich, Fon der Flaas, \cite{af2004}}]\label{thm-af}
For every $N\ge 6$ every $3$-coloring of~$[N]$ such that each color
class has cardinality at least
\[
|C|>
\begin{cases}
  \lfloor \frac{n+2}{6}\rfloor& \text{if $n\not\equiv 0 \pmod 6$;}\\
  \hfil\frac{n+4}{6} & \text{otherwise,}
\end{cases}
\]
contains a rainbow $3$-term arithmetic progression.
\end{theorem}

The lower bound on the cardinality  of the smallest color class  {in the above theorem turns out to be
best possible~\cite{JLMNR2001}}. {It is also shown in~\cite{af2004} that an analogous  result is no longer true for
arithmetic progressions of length $r\ge 5$ and $r$ colors, even if every color appears the same number
of times, and in Conlon, Jungi\'c and Radoi\v ci\'c~\cite{cjr2007}
$4$-colorings of $[n]$ were constructed with equinumerous
color classes without any $4$-term arithmetic progression whose elements are colored in distinct colors.}
On the other hand, such equinumerous colorings of $[rn]$ do
contain  rainbow arithmetic progressions of shorter length. For instance, Conlon, Fox and Sudakov~\cite{CFS2020}
have shown that there is a constant $c$ such that every equinumerous coloring of $[rn]$
with $r\ge ck^2\log k$ colors contains a rainbow $k$-term
arithmetic progression. If the condition of equinumerous coloring is dropped, the number of required colors
jumps to
$n^{1-o(1)}$ {for the existence of rainbow $k$-term arithmetic progressions with $k\ge 4$, see Butler
et al.~\cite{B2016}.}

{\subsection{Euclidean Canonical Ramsey} Canonical Ramsey theorems state the existence of some unavoidable
colored patterns  in every coloring. The Euclidean version has only been introduced recently,
see e.g.~Geh\'er,  Sagdeev and T\'oth~\cite{GST2025} and the references therein. Answering a question
in the latter  reference, it has been proved in Fang et al.~\cite{FGSXXY} that every coloring of $E^4$
contains a monochromatic or rainbow congruent copy of any given triangle, and  the existence of a
monochromatic or a rainbow congruent copy of a given rectangle
an $r$-coloring of $E^n$ for $n\geq n_0$ where $n_0$ depends on the rectangle.

\subsection{Gurevich Conjecture}
The problems that we address here are of similar nature, but related to the Gurevich conjecture on
the existence of a monochromatic triangle of a given area in {every finite coloring of the Euclidean plane}.
}

A stronger form of Gurevich's conjecture
was proved by R. L. Graham (\cite{gra2}, \cite{gra1}), which says that
 for any {real}  $\alpha >0$ and any pair of non-parallel
lines $L_1$ and~$L_2$, in any partition of the plane into finitely many
classes, some class contains the vertices of a triangle which has area
$\alpha$ and two sides parallel to the lines $L_i$. Later, a shorter proof, using the main idea of Graham,
was presented in~\cite{ad} and a very simple proof of the original conjecture was given by Dumitrescu and
Jiang~\cite{Jiang} by applying some clever induced coloring.

In~\cite{AC},  for any $\delta >0$, and any finite coloring of the plane,
existence of monochromatic trapezoids of area $\delta$ was established. However, it is known that the
answer to the
problem of the existence of monochromatic rectangle parallelograms is negative, while it is an open problem
for
general parallelograms.

In this note, we take up the corresponding problems related to  rainbow configurations and some canonical results. Among other results, we  show that every coloring of the plane lattice $\Z^2$ with $r\ge 3$ colors contains either a rainbow triangle of area $1/2$ or  monochromatic rectangles of every even area ({Theorem~\ref{t:rb-triangle-vs-even-mon-rects})}.
We also prove that there is an integer $A$ and a point $z\in \Z^2$ such that every coloring of $\Z^2$ contains  either a monochromatic parallelogram of area $A$
or a rainbow trapezoid of area $A$, with $z$ one of its vertices (Theorem \ref{mono-rainbow}).

\section{Rainbow triangles}

{Our first result is a canonical theorem for rainbow triangles and monochromatic rectangles of given area.}

{The following examples show that there are $r$-coloring of $\Z^2$  with $r\ge 3$ colors which avoid
rainbow triangles of area $1/2$ and $r$-colorings which avoid monochromatic rectangles with area $2$.
For the former,
color all of $\Z^2$ with a single color, except for the
points $(0,tm), \; m\in \Z$, to which we assign distinct colors.
The only possible rainbow triangles have area an integer multiple of $t/2$, which can be arbitrarily large so as to avoid any
fixed area, in particular $1/2$. For the latter, color each horizontal line with a single (unique) color such
that every set of three consectuive lines contains 3~distinct colors. The only monochromatic rectangles have area at least
three. The following result shows that no $r$-coloring with at least three colors can  avoid both structures.

 \begin{theorem}\label{t:rb-triangle-vs-even-mon-rects}
      In every $r$-coloring of~$\Z^2$ with $r\ge 3$ colors there is either
      a rainbow triangle of area~$1/2$ or monochromatic rectangles
      of each even area whose sides are parallel to the axes.
    \end{theorem}

    \begin{proof}
      Fix an $r$-coloring $\chi$ of~$\Z^2\!$.
      In what follows (unless explicitly stated otherwise) by a \emph{ line} we mean a horizontal line;
      by a \emph{rainbow triangle} we mean the vertices of a triangle of area~$1/2$
      whose colors are pairwise distinct;
      and by a \emph{monochromatic rectangle} we
      mean a set of points of the same color which form the vertices of a rectangle
      whose sides are parallel to the axes of~$\Z^2\!$.

      \medskip
      \emph{Case 1.} Some line $\ell$ contains points of $3$~distinct colors
      (i.e., $\bigl|\chi(\ell)\bigr|\ge 3$).
      By renaming the colors if necessary, we may assume that there is a  pair  of consecutive points
      in $\ell$ colored with~$\{a,b\}$ and a second pair of consecutive points colored with  $\{a',b'\,\}$
with $\{a,b\}\neq \{a',b'\,\}$.
      Then any occurrence of a color not in $\{a,b\}$ or not in $\{a',b'\,\}$  in a line adjacent to~$\ell$
      produces a rainbow triangle of area $1/2$ with the first or the second pair respectively.
      We conclude that if a rainbow triangle is not present, then both lines adjacent to $\ell$ receive uniquely a color in $\{a,b\}\cap \{c,d\}$. Since this intersection has size at most~$1$, all points in the lines adjacent to~$\ell$ are assigned the same color, yielding monochromatic rectangles of every even area.

      \medskip
      \emph{Case 2.} Some pair of consecutive lines $\ell_1,\, \ell_2$ contain points of at least $3$ distinct
colors
      (i.e., $\bigl|\chi(\ell_1\cup\ell_2)\bigr|\ge 3$ for some pair of consecutive lines~$\ell_1,\,\ell_2$).
      We may assume none of the lines contains 3~colors, for otherwise we are done by Case~1.
      Without loss of generality, we may also assume that $\ell_1$ contains adjacent points of colors~$a,\,b$,
      and $\ell_2$ contains a point of color~$c$. These three points form a rainbow triangle.

      \medskip
      \emph{Case 3.} There exist three consecutive lines $\ell_1,\,\ell_2$, and~$\ell_3$ that
      together contain at least~$3$ colors
      (i.e., $\bigl|\chi(\ell_1\cup\ell_2\cup\ell_3)\bigr|\ge 3$).
      We may assume that $\{a,\,b,\,c\}\subseteq \chi(\ell_1\cup\ell_2\cup\ell_3)$, that neither
      of the previous cases applies (in particular, $\bigl|\chi(\ell_i)\bigr|\le 2$ for each $i\in\{1,\,2,\,3\}$),
      and that $\ell_2$ lies between $\ell_1$ and~$\ell_3$.

      If each of these lines contains a single color, then we obtain a rainbow triangle.\footnote{{In fact,
      rainbow triangles of every  area multiple of $1/2$}.}
      Otherwise, one of these lines contains two colors.
      We may assume that $\bigl|\chi(\ell_2)\bigr|=1$, %: indeed, if $\bigl|\chi(\ell_2)\bigr|>2$ we are done by Case~1,
      as if $\bigl|\chi(\ell_2)\bigr|=2$ then either $\ell_1$ or~$\ell_3$ contains
      a third color and we are done by Case~2.
      In particular, we may suppose without loss of generality that $\chi(\ell_1)=\{a,b\}$, that $\chi(\ell_2)=\{a\}$ and $c\in\chi(\ell_3)$.

      Note that if $\chi(\ell_3)\not\subseteq \{a,c\}$,
      then $\ell_2$ and~$\ell_3$ together contain 3~colors, and we are done by Case~2.
      So $\chi(\ell_3)$ is either $\{c\}$ or~$\{a,c\}$. In the former case we clearly
      have rainbow triangles.\footnote{{Of every  area multiple of $1/2$.}}
      So suppose $\chi(\ell_3)=\{a,c\}$.
      We shall show that if we have no rainbow triangle of area~$1/2$ with vertices in~$\ell_1\cup\ell_2\cup\ell_3$,
      then we can find monochromatic rectangles.\footnote{Of every area.}
      Fix consecutive vertical lines $\nu_1$, $\nu_2$ such that $\chi(\nu_1\cap\ell_1)=a$ and $\chi(\nu_2\cap\ell_1)=b$.
      If~$\chi(\nu_2\cap \ell_3)=c$, then $\nu_2$ contains 3~colors and we are done by Case~1 (as its argument
      works replacing horizontal lines by vertical lines); moreover, we can not have $\chi(\nu_1\cap \ell_3)=c$ else we
      obtain a rainbow triangle with vertices {$\nu_2\cap\ell_1$, $\nu_2\cap \ell_2$} and $\nu_1\cap \ell_3$.
      It follows that {$\chi(\nu_1\cap\ell_i)=a$} for each $i\in\{1,\,2,\,3\}$.

      To conclude, fix $d\in\N$ and a vertical line $\nu'$ at distance $d$ from $\nu_1$;
      we shall find a monochromatic rectangle of area~$d$ or a rainbow triangle.
      If $\chi(\nu'\!\cap\ell_1)=a$, then $\nu_1\cap \ell_1,\,\nu'\!\cap \ell_1,\,\nu_1\cap \ell_2$ and
      ~$\nu'\!\cap \ell_2$ are the vertices
      of a monochromatic rectangle of area~$d$. And if $\chi(\nu'\cap \ell_1)=b$, then as before we must have $\chi(\nu'\cap \ell_3)=a$
      (as otherwise the line $\nu'$ contains~3 distinct colors and we are done by Case~1).
      Hence $\nu_1\cap\ell_2,\,\nu'\!\cap\ell_2,\,\nu_1\cap\ell_3$, and~$\nu'\!\cap\ell_3$ form a monochromatic rectangle of area~$d$ as required.

      \emph{Case 4.} None of the conditions in the three above cases are satisfied. Then every three consecutive lines contain only at most two colors.  Let $L$ be a maximal set of consecutive lines containing only two colors $a$ and $b$, say.  Then we find three consecutive lines $\ell_1,\,\ell_2$, and~$\ell_3$, such that $\ell_1,\ell_2\in L$, $\ell_3\not\in L$, and $\ell_3$ contains a third color, say~$c$. Since no set of three consecutive lines out of these three contains $3$~colors, we have that $\ell_1$ and~$\ell_2$ are monochromatic of the same color. Hence we find monochromatic rectangles of every integer area.
    \end{proof}
}

{Next results focus on the existence of only rainbow triangles of given area. }

\begin{proposition} \label{prop1} Given any $r$-coloring of the plane, with $r \geq 3$, such that the set
$\{(n,0): n \in \Z^+\}$ is not monochromatic, and the set $\{(n,2): n \in \Z^+\}$ receives at least three colors,
there are Rainbow triangles of area $1$.
\end{proposition}

\begin{proof} Since the set $\{(n,0): n \in \Z^+\}$ is not monochromatic, there are two points $A= (m,0)$, and~$B= (m+1, 0)$
such that $A,B$ receive different colors, say $c_1$ and $c_2$ respectively. Since
the set $\{(n,2): n \in \Z^+\}$ receives at least three colors, there is a point $(s,2)$ of some color other than
 $c_1$ and $c_2$, say $c_3$.
So, we get a Rainbow triangle $ABC$ of area $1$.
\end{proof}
\begin{remark}
It is clear that if at least two colors appear in
the set $\{(n,0): n \in \Z^+\}$ infinitely many times, or at least three colors appear infinitely many times in the set
$\{(n,2): n \in \Z^+\}$,
then there are infinitely many Rainbow triangle  of area $1$ between the parallel lines $y=0$ and~$y=2$.
\end{remark}

\begin{proposition}
For a large positive integer $N$, given any $r$-coloring of the plane, with $r \geq 3$, such that on the
set ${\bf S}= \{(n,0): n \in \Z^+,\, 1 \leq n \leq N\}$ at least three color classes have
at least $(N+4)/6$ points in  them and at least three colors appear on every line $y=k$, $k \geq 1$,
then there are points $A,B,C$ on the line $y=0$ and points $D,E$ on a line $y=t$ such that $ABD$ and $BCE$
are Rainbow triangles of~area~$N!$.
\end{proposition}

\begin{proof}
Let each of the color classes  $c_1,\,c_2,\,c_3$ appearing on ${\bf S}$ have at least $(N+4)/6$ points in  them.
We merge the remaining colors (if any) with the color $c_3$ to obtain a new color $c$.
By Theorem \ref{thm-af} of Axenovich and Fon der Flaass, under the $3$-coloring $c_1,\,c_2,\,c$ (and hence clearly under the
original coloring), the set ${\bf S}= \{(n,0): n \in \Z^+,\, 1 \leq n \leq N\}$ contains a rainbow
$3$-term arithmetic progression, say $A,\,B,\,C$. Let the lengths $AB$ and $BC$ be $s\leq \frac{N}{2}$.

Take $t=\frac{2N!}{s}$. Since at least three colors appear on  the line $L:y=t$, taking a point $D$ on $L$
having a color different from those of the points $A,\,B$ and taking a point $E$ on $L$
having a color different from those of the points $B,\,C$, we are through.
\end{proof}

\section{A canonical result for parallelograms and trapezoids}

{The next one is a canonical result for the problem of geometric configurations of given area. It is not difficult to
describe colorings which avoid monochromatic parallelograms and rainbow trapezoids. {For instance, let $D_r$ be a coloring of~$\Z^2$ such that each point $(x,y)$ gets color~$r$ where $r\equiv (y-x)\bmod 3$.
 It is easy to check that under $D_r$ the vertices of every rectangle of area~$2$ are colored using precisely $3$~colors.}
Therefore,  some condition has to be placed on the colorings.
  In view of Theorem \ref{thm-af},  we call a $3$-coloring of the integer lattice $\Z^2$ \emph{rainbow
feasible} if there is
an integer $N$ such that in every interval of length $N$ in every horizontal line each color appears at least $(N+4)/6$
times. In other words, we place the condition on the coloring that ensures the existence of rainbow $3$-term arithmetic
progressions.}

\begin{theorem} \label{mono-rainbow}
Let $\chi$ be a rainbow feasible coloring of the planar integer lattice $\Z^2$. There is an integer $A$
independent of the coloring and a point $z\in \Z^2$ such that there is either a monochromatic parallelogram of area~$A$
or a rainbow trapezoid of area~$A$, with $z$ one of its vertices.
\end{theorem}

\begin{proof} For each $(i,j)\in \Z^2$, let
\[
L_{ij}=\{(x,y)\in \Z^2: iN\le x \le (i+1)N-1, y=j\}
\]
be the segment of length $N$ in the horizontal line $y=j$ with first coordinate in the interval $[iN,(i+1)N-1]$. By
Theorem \ref{thm-af}, every $L_{ij}$ contains a rainbow  $3$-term arithmetic progression.  Each such rainbow $AP(3)$
in $L_{ij}$  can be identified by a triple $(x,d,\sigma)\in [N]\times [N/3]\times \text{Sym}(3)$ with  $iN+x$ its initial
point, $d$  its difference and $\sigma$ a permutation of $\{1,2,3\}$ such that $\chi (iN+x,j)=\sigma(1),
\chi (iN+x+d,j)=\sigma(2)$ and $\chi(iN+x+2d,j)=\sigma (3)$. Let $T=\{(x,d,\sigma)\in [N]\times [N/3]\times \text{Sym}(3)\}$
be the set of all possible triples  and denote by  $t=|T|=2N^2$ its cardinality.

Let $\chi_0$ be a coloring of the planar lattice by giving to $(i,j)$ a triple $(x,d,\sigma)$ in $T$ of a rainbow $AP(3)$
occurring in the interval $L_{i,j}$.
Define a coloring $\chi_1$ of $\N$ with the colors in $T$ where $\chi_1(j)$ is the set $\{\chi_0(i,j): i\in \N\}$ of triples
$(x,d,\sigma)$ appearing in rainbow $3$-term arithmetic progressions  in the segments $(L_{ij}: i\in \N)$. By van der
Waerden theorem there is a $\chi_1$-monochromatic arithmetic progression $P$ of length $t!+1$ in $[W]$ where $W$ is the
van der Waerden number $W=vdW(t!+1, 2^t-1)$. Let $d_P$ be the difference of $P$ and let $u$ be its initial term. define
\[
D=\frac{W!}{d_P}.
\]
Consider the $t+1$ segments
\[
L_{0,u}, L_{D,u}, \cdots , L_{tD, u}.
\]
By pigeonhole, two of them, say $L_{iD,u}, L_{jD,u}$,  have the same color $(x,d,\sigma)$ from the set $T$ under $\chi_0$.
Define
\[
m=\frac{t!}{j-i}.
\]
By the definition of $\chi_1$ and $P$, there is a segment $L_{k,u+md_P}$ in the horizontal line $L_{u+md_P}$ with the same
color $(x,d,\sigma)$ under $\chi_0$ as the two segments above. Denote by $(\alpha_1,\alpha_2,\alpha_3)$,
$(\beta_1,\beta_2,\beta_3)$ and $(\gamma_1,\gamma_2,\gamma_3)$  the three rainbow arithmetic progressions of
type $(x,d,\sigma)$  in the segments $L_{iD,u}, L_{jD,u}$ and $L_{k,u+md_P}$ respectively.

Let $z$ be the point $z=(kN+x+d,u+md_P)$ in the same line.

If $\chi (z)=\sigma (2)$ then the four points
\[
\{\alpha_2,\beta_2,\gamma_2,z\}
\]
form a $\chi$-monochromatic parallelogram with area
\[
A=D(j-i)md_P=W!k!.
\]
Otherwise, if $\chi (z)\neq \sigma(2)$, then the four points
\[
\{\alpha_1,\beta_2,\gamma_3, z\}
\]
form a trapezoid with three vertices of different colors and area
\[
A=D(j-i)md_P=W!k!.
\]
\begin{figure}[h]
\begin{tikzpicture}[scale=1]
\draw (0,0)--(10,0);
\draw (0,4)--(10,4);
\foreach \i in {1,2,3}
{
\draw[fill] (1+\i,0) circle(2pt);
\draw[fill] (5+\i,0) circle(2pt);
\draw[fill] (2+\i,4) circle(2pt);
}
\draw[fill] (8,4) circle(2pt);
\node[above] at (3,4) {\footnotesize{$\alpha_1$}};
\node[above] at (4,4) {\footnotesize{$\alpha_2$}};
\node[above] at (5,4) {\footnotesize{$\alpha_3$}};
\node[above] at (8,4) {\footnotesize{z}};

\node[below] at (2,0) {\footnotesize{$\beta_1$}};
\node[below] at (3,0) {\footnotesize{$\beta_2$}};
\node[below] at (4,0) {\footnotesize{$\beta_3$}};
\node[below] at (6,0) {\footnotesize{$\gamma_1$}};
\node[below] at (7,0) {\footnotesize{$\gamma_2$}};
\node[below] at (8,0) {\footnotesize{$\gamma_3$}};

\draw[dashed] (3,0)--(4,4)--(8,4)--(7,0)--(3,0);
\draw[dotted] (3,0)--(3,4)--(8,4)--(8,0)--(3,0);

\end{tikzpicture}
\caption{An illustration of the appearance of the monochromatic or rainbow structure in the proof of Theorem \ref{mono-rainbow}.}
\label{fig:parallelogram-trapezoid}
\end{figure}
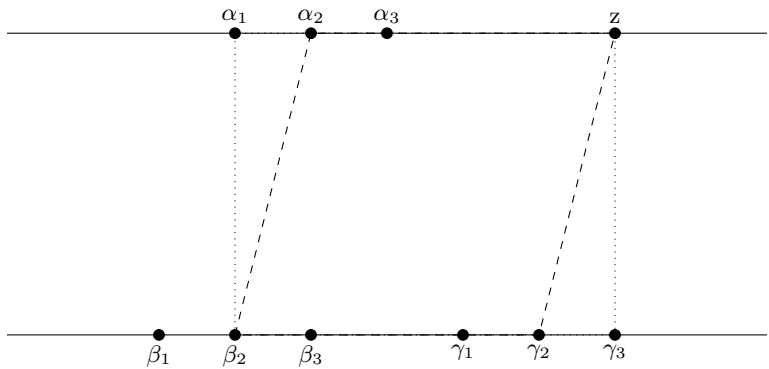
This completes the proof (see Figure~\ref{fig:parallelogram-trapezoid}).
\end{proof}

\begin{proof} (Second proof using Graham's theorem)
For each $(i,j)\in \Z^2$, let
\[
L_{ij}=\{(x,y)\in \Z^2: iN\le x \le (i+1)N-1, y=j\}
\]
be the segment of length $N$ in the horizontal line $y=j$ with first coordinate in the interval $[iN,(i+1)N-1]$.
By Theorem \ref{thm-af}, every $L_{ij}$ contains a rainbow  $3$-term arithmetic progression.  Each such rainbow
$AP(3)$ in $L_{ij}$  can be identified by a triple $(x,d,\sigma)\in [N]\times [N/3]\times \text{Sym}(3)$ with
$iN+x$ its initial point, $d$  its difference and $\sigma$ a permutation of $\{1,2,3\}$ such that $\chi (iN+x,j)=\sigma(1),
\chi (iN+x+d,j)=\sigma(2)$ and $\chi(iN+x+2d,j)=\sigma (3)$. Let $T=\{(x,d,\sigma)\in [N]\times [N/3]\times \text{Sym}(3)\}$
be the set of all possible triples  and denote by  $t=|T|=2N^2$ its cardinality.

Let $\chi_0$ be a coloring of the planar lattice by giving to $(i,j)$ a triple $(x,d,\sigma)$ in $T$ of a rainbow $AP(3)$
occurring in the interval $L_{i,j}$. By Graham's theorem there is a monochromatic triangle with given area $A$. Moreover
we may assume that two sides of the triangle are parallel to the axes. Now we complete the argument as before by considering
the point which completes a rectangle with the second element of each rainbow $AP(3)$ if the color agrees with them, or a
trapezoid of the same area with the points switched.
\end{proof}

For our next result, we shall require some stronger conditions.
Apart from assuming that the given $r$-coloring ($r \geq 4$)
at least three color classes have at least $(n+4)/6$ points on the line $y=0$, we shall assume that
on each infinite arithmetic progression (with any positive real number as the common difference)
 on a line $y=k$, with $k \geq 1$, at least four colors appear.
As will be clear from the proof of our result  we can replace the lines
by a suitable collection of parallel lines.

\begin{proposition}
For every even $N\ge 6$, given any $r$-coloring of the plane, with $r \geq 4$,
such that on the
set ${\bf S}= \{(n,0): n \in \Z^+\, 1 \leq n \leq N\}$, at least three color classes have
at least $(N+4)/6$ points in  them and such that each infinite arithmetic progression on a horizontal line $y=k$, with $k \geq 1$,
at least four colors appear, then there is either a Rainbow parallelogram
of area $N!$ or a Rainbow trapezium of area~$\frac{3}{2}N!$.
\end{proposition}

\begin{proof}
On the set ${\bf S}= \{(n,0): n \in \Z^+\!,\, 1 \leq n \leq N\}$,
if each of the color classes  $c_1,\,c_2,\,c_3$  has at least $(N+4)/6$ points in  them, as before
we merge the remaining colors with the color $c_3$ to obtain a new color $c$.
By Theorem \ref{thm-af} of Axenovich and Fon der Flaass, under the $3$-coloring $c_1,\,c_2,\,c$ (and hence clearly under the
original coloring), the set ${\bf S}$  contains a Rainbow
$3$-term arithmetic progression, $A,\,B,\,C$ with three colors $c'_1,\,c'_2,\,c'_{3}$ respectively.
Let $d\leq \frac{N}{2}$ be the common difference of the AP $A,\,B,\,C$.

Considering an infinite AP with common difference $d$ on the line $y=t$ where $t=\frac{N!}{d}$,
and a fourth color say $c_4$ appearing on it,
there are two points, say $E$ and~$F$,  at a distance $d$ on $y=t$ with different colors,
one of them having color $c_4$.

If the other color happens to be $c'_1$ or $c'_3$, we shall have a Rainbow parallelogram $BCEF$ or $ABEF$
of area $N!$ and if the other color happens to be $c'_2$, we have a Rainbow trapezium $ACEF$ of area
$\frac{3}{2}N!$.
\end{proof}

\smallskip

We finally observe that, in searching for a canonical Ramsey theorem for fixed-area rectangles in~$\Z^2$,
it is not possible to guarantee either monochromatic or rainbow rectangles, at least when this area is not a multiple of~$4$.

\begin{proposition}
  For every integer $A$ that is not divisible by~$4$, and every $r\ge 2$,
  there exists an $r$-coloring of~$\Z^2$ that avoids
  both monochromatic and rainbow rectangles of area~$A$ whose sides
  are parallel to the axes.
\end{proposition}

In the following proof a rectangle means a collection of four distinct elements of~$\Z^2$
which are the vertices of a rectangle with sides parallel to the axes.

\begin{proof}
  Consider the following coloring
  \[
    \chi(x,y) =
    \begin{cases}
      \hfil r-1            & \text{if $x + y$ is even,} \\
      \displaystyle x+y - (r-1)\left\lfloor\frac{x+y}{(r-1)}\right\rfloor   & \text{otherwise}
    \end{cases}
  \]
  and let $R=\{(x,y),\,(x+w,y),\,(x,y+h),\,(x+w,y+h)\}$ be an arbitrary
  rectangle with integer coordinates and area~$A=hw$.
  We claim that precisely two vertices of~$R$ get color~$r-1$.

  Note that the only points receiving color~$r-1$ under~$\chi$ are those whose coordinates' sum is even.
  Moreover, since $A\not\equiv 0\pmod 4$, we have that $h$ and~$w$ cannot both be even.
  So, suppose without loss of generality that~$h$ is odd; this means that either
  $w$ is even or $h+w$ is even. In other words, either
  \[
    \begin{cases}
      x+y   & \equiv x+y+w   \\
      x+y+h & \equiv x+y+w+h
    \end{cases}
    \qquad\text{or}\qquad
    \begin{cases}
      x+y   & \equiv x+y+w+h \\
      x+y+h & \equiv x+y+w
    \end{cases}
  \]
  where all equivalences are mod~$2$ and
  the last inequality is immediate from the fact that $h+w\equiv w-h \pmod 2$. By the definition of~$\chi$, we
  have that each rectangle of area~$s$ must therefore be assigned either $2$ or~$3$ colors by~$\chi$ as required.
\end{proof}

It is easy to check that many other colorings achieve a similar effect.
(For instance, we may assign color `red' to points in~$\Z^2$ of even coordinate-sum,
and choose an arbitrary coloring of the remaining points using colors
other than `red'.)

\end{document}